\documentclass{amsart}
\usepackage{amsthm,amsfonts,amsmath,amssymb,latexsym,epsfig,mathrsfs,yfonts,marvosym}
\usepackage[usenames]{color}

\DeclareMathAlphabet\oldmathcal{OMS}        {cmsy}{b}{n}
\SetMathAlphabet    \oldmathcal{normal}{OMS}{cmsy}{m}{n}
\DeclareMathAlphabet\oldmathbcal{OMS}       {cmsy}{b}{n}
 
\usepackage{eucal}

\usepackage{graphicx}
\usepackage[all]{xy}
\usepackage{epsfig}

\newtheorem{theorem}{Theorem}[section]
\newtheorem{definition}[theorem]{Definition}

\newtheorem{lemma}[theorem]{Lemma}

\newtheorem{crit}[theorem]{Criterion}

\newtheorem{defn-thm}[theorem]{Definition--Theorem}  
\newtheorem{defn-prop}[theorem]{Definition--Proposition} 

\theoremstyle{definition}

\newtheorem{example}[theorem]{Example}

\newtheorem{remark}[theorem]{Remark}
\newtheorem{ack}{Acknowledgments\!\!\!}

\theoremstyle{remark}

\newcommand{\be}{\begin{equation}}
\newcommand{\ea}{\end{array}}\newcommand{\ee}[1]{\label{#1}\end{equation}}

\newcommand{\mult}[0]{\operatorname{mult}}

\newcommand{\hl}{\\\hline}

\def\fract#1#2{\raise4pt\hbox{$ #1 \atop #2 $}}
\def\decdnar#1{\phantom{\hbox{$\scriptstyle{#1}$}}
\left\downarrow\vbox{\vskip15pt\hbox{$\scriptstyle{#1}$}}\right.}

\def\bfa{{\bf a}}

\def\bfw{{\bf w}}

\def\calo{{\mathcal O}}

\def\cald{{\mathcal D}}

\def\calf{{\mathcal F}}

\def\cali{{\mathcal I}}

\def\calo{{\mathcal O}}

\def\cals{{\oldmathcal S}}
\def\calS{{\mathcal S}}

\def\calz{{\oldmathcal Z}}

\def\bbc{{\mathbb C}}

\def\bbn{{\mathbb N}}

\def\bbp{{\mathbb P}}
\def\bbq{{\mathbb Q}}

\def\bbz{{\mathbb Z}}

\def\grg{\gamma}

\def\grl{\lambda}

\def\grz{\zeta}
\def\grD{\Delta}

\def\grG{\Gamma}

\def\tU{\tilde{{U}}}

\def\gsp1{{\mathfrak s}{\mathfrak p}(1)}

\def\la#1{\hbox to #1pc{\leftarrowfill}}
\def\ra#1{\hbox to #1pc{\rightarrowfill}}
\def\Se{Sasaki-Einstein }

\def\BOne{{\mathchoice {\rm 1\mskip-4mu l} {\rm 1\mskip-4mu l}
                          {\rm 1\mskip-4.5mu l} {\rm 1\mskip-5mu l}}}
\def\hook{\mathbin{\hbox to 6pt{%
                 \vrule height0.4pt width5pt depth0pt
                 \kern-.4pt
                 \vrule height6pt width0.4pt depth0pt\hss}}}

\DeclareMathAlphabet\oldmathcal{OMS}        {cmsy}{b}{n}
\SetMathAlphabet    \oldmathcal{normal}{OMS}{cmsy}{m}{n}
\DeclareMathAlphabet\oldmathbcal{OMS}       {cmsy}{b}{n}

\begin{document}
\bibliographystyle{amsalpha}

\title{On Sasaki-Einstein Manifolds in Dimension Five}\thanks{During the preparation of this
work the first author
was partially supported by NSF grant DMS-0504367.}

\author{Charles P. Boyer and Michael Nakamaye}
\address{Department of Mathematics and Statistics,
University of New Mexico, Albuquerque, NM 87131.}

\email{cboyer@math.unm.edu} 
\email{nakamaye@math.unm.edu}

\begin{abstract}
We prove the existence of Sasaki-Einstein metrics on certain simply connected 5-manifolds where until now existence was unknown. All of these manifolds have non-trivial torsion classes. On several of these we show that there are a countable infinity of deformation classes of Sasaki-Einstein structures.
\end{abstract}

\maketitle

\section{Introduction}

Sasakian geometry has recently proven to be a rich source for the production of positive Einstein metrics. In \cite{BG01b} a method for obtaining such Sasaki-Einstein metrics was introduced and developed in much more detail in \cite{BGN02a,BGN02b,BGN03c,BG03,BGK05,BGKT05,BG06b,BoGa05a}. In particular, dimension 5 is quite tractable due the Smale-Barden \cite{Sm62,Bar65} classification of simply connected 5-manifolds. However, it is well-known that any simply connected Sasaki-Einstein manifold must be spin, so Smale's classification of simply connected spin 5-manifolds is what is needed for us. We refer to such manifolds as {\it Smale manifolds}.  It is convenient to divide the Smale manifolds into 3 classes: the simply connected rational homology spheres, including $S^5$; the torsion free Smale manifolds with positive second Betti number, that is the $k$-fold connected sums of $S^2\times S^3$ with $k\geq 1$; and the mixed type, that is, connected sums of torsion free Smale manifolds with simply connected rational homology spheres. 

An important question in Riemannian geometry involves varying notions of positive curvature. In order of decreasing strength they are: positive sectional curvature, positive Ricci curvature, and positive scalar curvature. In particular, it is of much interest to determine which manifolds admit or cannot admit Riemannian metrics of positive Ricci curvature. It follows from the transverse version of the Calabi-Yau theorem \cite{ElK,BGN03a,BGN03b} that positive Sasaki-Seifert structures give Sasakian metrics of positive Ricci curvature. Of course, since any Sasaki-Einstein metric has positive Einstein constant, they automatically give Riemannian metrics of positive Ricci curvature. We are, therefore, interested in the following:
\newline{\bf Problems}:
 
\begin{enumerate}
\item Determine precisely which Smale manifolds admit positive Sasakian structures. 
\item Determine the cardinality of the set of positive Sasaki-Seifert structures on each Smale manifold.  
\item Determine which Smale manifolds admit Sasaki-Einstein metrics. 
\item Determine the cardinality of the set of deformation classes of Sasaki-Einstein structures. 
\end{enumerate}

Note that (3) is a subset of (1) and (4) of (2).
All four problems have been completely solved for $k(S^2\times S^3)$ \cite{BGN03a,BGN03c,Kol04}. In the case of simply connected rational homology spheres, problem (1) has been solved completely \cite{Kol05b}, problem (2) is partially solved, and problem (3) is done with the exception of the manifolds $nM_2$ for $n\geq 2$ \cite{BG06b,Kol05b,BG05}. Moreover, $nM_2$ admits infinitely many positive Sasaki-Seifert structures for each $n>0$ which can be realized by links of weighted homogeneous polynomials (see Theorem 10.2.9 of \cite{BG05}). 
Recall that Koll\'ar \cite{Kol05b} has shown that in order for a simply connected 5-manifold to admit a positive Sasakian structure very stringent conditions on the torsion subgroup of $H_2(M,\bbz)$ must hold. Specifically
\begin{theorem}[Koll\'ar]\label{Kol05}
Let $M$ be a simply connected 5-manifold and suppose that $M$ admits a positive Sasakian structure. Then $H_2(M,\bbz)_{tor}$ must be one of the following groups:
$$\bbz_m^2,\bbz_2^{2n},\bbz_3^4,\bbz_3^6,\bbz_3^8,\bbz_4^4,\bbz_5^4~{\rm with}~n>1~ {\rm and}~m\geq 1.$$
In particular, if $M$ admits a Sasaki-Einstein metric, then $H_2(M,\bbz)_{tor}$ must be one of these groups. (Note that $m=1$ means zero torsion.) 
\end{theorem}

Until now what is known about the existence of Sasaki-Einstein metrics as well as positive Ricci curvature Sasakian metrics on Smale manifolds appears in Chapters 10 and 11 of \cite{BG05}. In fact, Corollary 11.4.14 of \cite{BG05} gives a list of all Smale manifolds that can possibly admit a Sasaki-Einstein metric, and indicates those for which existence is known. It is the purpose of this note to present some new results on the existence of Sasaki-Einstein metrics on simply connected 5-manifolds. In particular, we show that the answer to parts (iv) and (v) of Question 11.4.1 of \cite{BG05} is in the affirmative, but so far we have been unable to determine all the manifolds that can occur. Specifically, we prove

\newtheorem*{named}{\bf Theorem A}
\begin{named}\label{rhsSE}
{\it We have:
\begin{enumerate}
\item The rational homology spheres 
$$2M_2, \quad 3M_2, \quad 5M_2, \quad 6M_2, \quad 7M_2$$
admit Sasaki-Einstein metrics.
\item {\it There exist Sasaki-Einstein metrics on the mixed Smale manifolds
$$M_\infty\#M_3, ~M_\infty\#M_4, ~M_\infty\#M_5, ~M_\infty\#M_6, ~M_\infty\#M_7,~M_\infty\#2M_3,M_\infty\#3M_3,$$ 
$$~2M_\infty\#M_2,~2M_\infty\#M_3, ~2M_\infty\#M_4, ~2M_\infty\#M_5,~3M_\infty\#M_2,~3M_\infty\#M_3,~3M_\infty\#M_4,$$
$$4M_\infty\#M_2, ~4M_\infty\#M_3, ~4M_\infty\#2M_2, ~5M_\infty\#M_2, ~5M_\infty\#2M_2, ~6M_\infty\#M_2.$$
\item Moreover, the manifolds $2M_3,M_\infty\#M_3,M_\infty\#2M_3$ and $3M_\infty\#M_3$ admit a countable infinity of Sasaki-Seifert structures each having Sasaki-Einstein metrics}.
\end{enumerate}}
\end{named}

Compared with our earlier work \cite{BGN02a,BGN02b,BGN03c,BG03,BGK05,BGKT05,BG06b}, the proof of this theorem requires two new ingredients, first, Koll\'ar's Theorem \ref{Kol05} \cite{Kol05b}, and second, a generalization of previous estimates used in \cite{BGN03c} to the case when branch divisors are present. The latter is Lemma \ref{klt5} below.

The manifolds in part (3) of Theorem A are the first examples of an infinite number of deformation classes of Sasaki-Einstein metrics on Smale manifolds with torsion. The question as to which Smale manifolds admit positive Sasaki-Seifert structures that can be represented by links of isolated hypersurface singularities by weighted homogeneous polynomials is in itself an interesting one. Indeed, looking at the table in Theorem 10.2.25 of \cite{BG05} we see that all Smale manifolds of the form $kM_\infty\# M_m$ with $k=1,\cdots, 8,$ and $m\geq 12$ admit positive Sasaki-Seifert structures represented by isolated hypersurface singularities by weighted homogeneous polynomials except $5M_\infty\# M_m$. Why this happens is not understood. Furthermore, on the manifolds $6M_\infty\# M_m$, $7M_\infty\# M_m$, $8M_\infty\# M_m$ with $m\geq 12$ there is a unique positive Sasaki-Seifert structure and it can be represented by such a link. Generally, it seems to be not very well understood exactly why and when positive Sasaki-Seifert structures can be represented by links of isolated hypersurface singularities by weighted homogeneous polynomials. It seems quite remarkable that most of the known positive Sasaki-Seifert structures can be represented by such links. Keep in mind, however, that the complete resolution of problem (1) above is still open, and that the present paper treats only 10 of the 19 normal forms obtained by Yau and Yu \cite{YaYu05} for links of weighted homogeneous polynomials, so exactly how effective our method is, remains to be seen. Furthermore, there is a generalization of our methods to links of complete intersections of weighted homogeneous polynomials where almost nothing is known.

For some reason there are gaps appearing in the types of manifolds where our method is successful. For example, $4M_2$ is proving elusive as far as the existence of Sasaki-Einstein metrics. We have found several infinite series of positive Sasaki-Seifert structures on $4M_2,$ including the double infinite series $\bfw=(2l,2(2k+1),(9l-1)(2k+1),9l(2k+1))$ of degree $d=18l(2k+1)$ with $\gcd(l,2k+1)=1.$ The index is $I=2l+2k+1$. No Sasaki-Einstein metrics are known to exist on  these links, and for $2k>4l-1$ or $l>5k+2$ no Sasaki-Einstein metrics can exist by the Lichnerowicz obstruction described in \cite{GMSY06}. Again the exact reason seems elusive.

\newtheorem*{named2}{\bf Theorem B}
\begin{named2}\label{mixedSE}
{\it In addition to the Smale manifolds described in Section 10.2.1 of \cite{BG05} and those of Theorem A,  the following Smale manifolds admit a countable infinity of positive Sasaki-Seifert structures; hence, each such Sasaki-Seifert structure has Sasakian metrics of positive Ricci curvature:
$$M_4,M_6,M_\infty\#M_4,2M_\infty\#4M_2,2M_\infty\#M_3, 3M_\infty\#M_2, 3M_\infty\#3M_2,6M_\infty\#M_2,$$
and $M_\infty\#nM_2$ for each $n\geq 2.$ 
Moreover, $2M_\infty\#nM_2$ admits Sasakian metrics of positive Ricci curvature for each $n>0$.}
\end{named2}

\begin{remark}
From Theorem B and Table B.4.1 of \cite{BG05}, we see that the rational homology 5-spheres $M_m$ for $m=2,\cdots,6$ all admit a countable infinity of positive Sasaki-Seifert structures, whereas, for $m\geq 7$ only finitely many are known. Indeed, from the discussion on page 345 of \cite{BG05} it was anticipated that $M_4$ admits a countable infinity of positive Sasaki-Seifert structures, but not for $M_6.$ Moreover, we now know that $M_7$ admits at least 19 positive Sasaki-Seifert structures, but whether this number is finite or not is still unknown. 
\end{remark}

\begin{ack}
The authors thank Ivan Cheltsov for pointing out a gap in our previous paper as well as his interest in our work.
The first author would like to thank the Courant Institute of Mathematical Sciences of New York University for its hospitality during which this work was finished.
\end{ack}

\section{Sasakian Structures on Links of Isolated Hypersurface
Singularities}

We begin with a very brief summary of the Sasakian geometry of links of isolated
hypersurface singularities defined by weighted homogeneous polynomials.
For more details we refer the reader to \cite{BG05}.
Consider the affine space $\bbc^{n+1}$ together with a weighted
$\bbc^*$-action $\bbc^*_\bfw$ given by $(z_0,\ldots,z_n)\mapsto
(\grl^{w_0}z_0,\ldots,\grl^{w_n}z_n),$ where the {\it weights} $w_j$ are
positive integers. It is convenient to view the weights as the components of a
vector $\bfw\in (\bbz^+)^{n+1},$ and we shall assume that
$\gcd(w_0,\ldots,w_n)=1.$ Let $f$ be a quasi-homogeneous polynomial, that is
$f\in \bbc[z_0,\ldots,z_n]$ and satisfies
$$f(\grl^{w_0}z_0,\ldots,\grl^{w_n}z_n)=\grl^df(z_0,\ldots,z_n),
$$
where $d\in \bbz^+$ is the degree of $f.$ We are interested in the {\it
weighted affine cone} $C_f$ defined by
the equation $f(z_0,\ldots,z_n)=0.$ We shall assume that the origin in
$\bbc^{n+1}$ is an isolated singularity, in fact the only singularity, of
$f.$ Then the {\it link} $L_f$ defined by
$$L_f= C_f\cap S^{2n+1}, $$
where
$$S^{2n+1}=\{(z_0,\ldots,z_n)\in \bbc^{n+1}|\sum_{j=0}^n|z_j|^2=1\}$$
is the unit sphere in $\bbc^{n+1},$ is a smooth manifold of dimension $2n-1.$
Furthermore, by the Milnor Fibration Theorem it is well-known that the link $L_f$ is $(n-2)$-connected.

An important special case of weighted homogeneous polynomials, which was used to great advantage in \cite{BGK05}, is that of the {\it Brieskorn-Pham (BP) polynomials}. These take the form
$$f(z_0,\cdots,z_n)=z_0^{a_0}+\cdots + z_n^{a_n},$$
where the exponents $a_j$ and weights $w_j$ are related by $a_jw_j=d$ for all $j=0,\cdots,n.$ The BP polynomials are the simplest type whose zero set has only an isolated singularity at the origin. There is a classification of weighted homogeneous polynomials in terms of Newton polytopes \cite{Kou76}, and a list of normal forms up to deformation has been given for $n=2$ (3-dimensional links) by Orlik and Wagreich \cite{OrWa71a,OrWa77}, and for $n=3$ (5-dimensional links) by Yau and Yu \cite{YaYu05}. It is always assumed that the polynomial is {\it quasi-smooth}, that is, the hypersurface in $\bbc^{n+1}$ cut out by $f=0$ has only an isolated singularity at the origin. The criterion for quasi-smoothness is given in \cite{Fle00} (see also Theorem 4.6.16 of \cite{BG05}). Since this paper only deals with five dimensional links, we give this case only, viz.: 
\begin{crit}
A hypersurface
$\calz_d\subset\bbc\bbp(w_0,w_1,w_2,w_3)$ is quasi-smooth if and only
if all of the following three conditions hold 
\begin{enumerate}
\item For each $i=0,\ldots,3$ there is a $j$ and a monomial
$z_i^{m_i}z_j\in H^0(\bbc\bbp(\bfw),\calo(d)).$ Here $j=i$ is
possible.

\item If $\gcd(w_i,w_j)>1$ then there is a monomial
$z_i^{b_i}z_j^{b_j}\in H^0(\bbc\bbp(\bfw),\calo(d)).$

\item For every $i,j$ either there is a monomial
$z_i^{b_i}z_j^{b_j}\in H^0(\bbc\bbp(\bfw),\calo(d)),$ or there are
monomials $z_i^{c_i}z_j^{c_j}z_k$ and $z_i^{d_i}z_j^{d_j}z_l\in
H^0(\bbc\bbp(\bfw),\calo(d))$ with $\{k,l\}\neq \{i,j\}.$
\end{enumerate} 

\end{crit}

The list of normal forms given in \cite{YaYu05} consists of 19 types, referred to herein as {\it YY types}. It was reproduced in an appendix of \cite{BG05}, and we also reproduce it here in the appendix for the convenience of the reader. However, only 10 of the 19 types of Yau-Yu links can be written as a linear system of 4 equations in 4 unknowns. These are types I-VII, X,XI, and XIX, and are the ones we treat in this paper.

We say that a polynomial $f'$ is a {\it perturbation} of a polynomial $f$ if a monomial of the form $z_i^{a_i'}z_j$ in $f'$ with $j\neq i$ can be replaced by a monomial of the form $z_i^{a_i}$ to give $f.$ In such a case the polynomials $f$ and $f'$ have the same degree $d$ and weight vector $\bfw.$ It means that the weight $w_i$ of $f'$ divides the degree $d.$ After making all such replacements we say that $f$ {\it minimizes} $f'$ (or that $f$ has {\it minimal} YY type). Once a minimal polynomial is given any additional monomials of the form $z_0^{b_0}z_1^{b_1}z_2^{b_2}z_3^{b_3}$ with $\sum_ib_iw_i=d$ can be added to $f$ to give the most general polynomial with a given weight vector $\bfw$ and degree $d.$ One can then give a local moduli space of such links. In the case that Sasaki-Einstein metrics exist one then gets a local moduli space of Sasaki-Einstein metrics. See Section 5.5.2 of \cite{BG05} for details.

Recall that a Sasakian structure consists of a quadruple $\cals =(\xi,\eta,\Phi,g)$ where $g$ is 
a Riemannian metric, $\xi$ is a unit length Killing vector field, $\eta$ is a contact 1-form such 
that $\xi$ is its Reeb vector field, and $\Phi$ is a $(1,1)$ tensor field which annihilates $\xi$ 
and describes an integrable complex structure on the contact vector bundle $\cald 
=\hbox{ker}~\eta$ with transverse metric $d\eta\circ (\Phi\otimes \BOne).$ 
On $S^{2n+1}$ there is a well-known  ``weighted'' Sasakian structure
$\cals_\bfw=(\xi_\bfw,\eta_\bfw,\Phi_\bfw,g_\bfw)$, where the vector field
$\xi_\bfw$ is the infinitesimal generator of the circle subgroup $S^1_\bfw
\subset \bbc^*_\bfw.$ This Sasakian structure on $S^{2n+1}$ induces a natural
Sasakian structure, also denoted by
$\cals_\bfw$, on the link $L_f.$ We shall often write $L_f=L(\bfw;d)$ to emphasize the dependence on the weight vector and degree. 
The quotient space $\calz_f$ of $S^{2n+1}$ by the circle $S^1_\bfw,$ or
equivalently the space of leaves of the characteristic foliation $\calf_\xi$
of $\cals_\bfw,$  is a projective
algebraic variety with an orbifold structure embedded in the weighted projective $\bbp(\bfw)=
\bbp(w_0,w_1,\cdots,w_n),$ in such a way that there is a commutative diagram
\begin{equation}
\begin{matrix}
 L_f &\ra{2.5}& S^{2n+1}_\bfw& \\
  \decdnar{\pi}&&\decdnar{} &\\
   \calz_f &\ra{2.5} &\bbp(\bfw),&
\end{matrix}
\end{equation}
where the horizontal arrows are Sasakian and K\"ahlerian embeddings,
respectively, and the vertical arrows are principal $S^1$ orbibundles and
orbifold Riemannian submersions.  

We are interested in the set of all Sasakian structures associated to the foliation $\calf_\xi.$ When the foliation is quasi-regular, this is what we  have called a {\it Sasaki-Seifert} structure \cite{BG05}, and is conveniently phrased in terms of the basic cohomology class $[d\eta]_B\in H^{1,1}(\calf_\xi)$ via
\begin{equation}\label{sasset}
{\mathcal S}^+(\calf_\xi)=\{\cals'=(a\xi,\eta',\Phi',g')~|~[d\eta']_B=a^{-1}[d\eta]_B,~ a>0  \}.
\end{equation}
Its subset $\calS(\xi)$ defined by putting $a=1$ in (\ref{sasset}) can also be described in terms of the deformations $\eta\mapsto \eta_t=\eta +\grz_t$ where $\grz_t$ is basic, $\eta_0=\eta,~\grz_0=0$ and $\eta_t\wedge (d\eta_t)^n\neq 0$
for all $t\in [0,1].$ The {\it first basic Chern class} $c_1(\calf_\xi)$ is an important invariant of the space $\calS(\calf_\xi)$, and when  $c_1(\calf_\xi)$ can be represented by a positive definite $(1,1)$-form any Sasakian structure $\cals\in \calS(\calf_\xi)$ is said to be {\it positive}. For the weighted homogeneous links $L_f$ this corresponds to $I=\sum w_i-d>0$, and the quotient $\calz_f$ being a Fano orbifold. 

Suppose now we have  a link $L_f$ with a given
Sasakian structure $(\xi,\eta,\Phi,g)$. When can we find a 1-form $\grz$ such
that the deformed structure $(\xi,\eta+\grz,\Phi',g')$ is Sasaki-Einstein? This is a
Sasakian version of the Calabi problem and its solution is equivalent to
solving the corresponding Calabi problem on the space of leaves $\calz_f$.
Since a \Se manifold necessarily has positive Ricci tensor, its
Sasakian structure is necessarily positive. This also implies that the
K\"ahler structure on the orbifold $\calz_f$ be positive, i.e. $c_1^{orb}(\calz_f)$ can be
represented by a positive definite $(1,1)$ form. In this case there are
well-known obstructions to solving the Calabi problem. In local uniformizing coordinates $\{z_i\}_i$ for the complex orbifold $\calz_f$, one uses the continuity method to find solutions to the Monge-Amp\`ere equations
\begin{equation}\label{MAteqn}
{\rm det}\Bigl(g_{i\bar{j}}+
\frac{\partial^2\varphi}{\partial z_i\partial\bar z_j}\Bigr)
={\rm det}(g_{i\bar{j}})e^{F-t\varphi}
\end{equation}
for all $t\in [0,1]$ subject to the condition

\begin{equation}\label{MAteqn2}
g'_{i\bar{j}}=g_{i\bar{j}}+\frac{\partial^2\varphi}{\partial z_i\partial\bar z_j}>0,
\end{equation}
where $F$ is the {\it discrepancy potential}. See, for example, Section 5.2 of \cite{BG05} and references therein. Generally there are obstructions to solving Equation (\ref{MAteqn}) for all $t.$ What is needed is a uniform bound. The continuity method will produce a K\"ahler-Einstein orbifold metric on $\calz_f,$ hence, a \Se metric on $L_f,$ if there is a
$\gamma>\tfrac{n}{n+1}$ such that for every $s\geq 1$ and for
every holomorphic section $\tau_s\in H^0(\calz_f,
\calo(-sK_{\calz_f}^{orb}))=\calo(s(|\bfw|-d))$ we have $|\tau_s|^{-\frac{\grg}{s}}\in
L^2(\calz_f),$ that is
\begin{equation}\label{l2eqn}
\int  |\tau_s|^{-\frac{2\gamma}{s}}d{\rm vol}_g <\  +\infty\,.
\end{equation}
Here $K_{\calz_f}^{orb}$ denotes the {\it orbifold canonical divisor} which generally differs from the ordinary canonical divisor $K_{\calz_f}$ by a branch divisor. Recall that a {\it branch divisor} $\grD$ is a $\bbq$-divisor of the form $\grD=\sum_i(1-\frac{1}{m_i})D_i$ where $D_i$ are Weil divisors on the underlying algebraic variety of $\calz_f$, and $m_i$ is the {\it ramification index}, that is, the gcd of the orders of the local uniformizing groups $\grG$ of the orbifold $\calz_f$ at all points of $D_i.$

Fundamental to this method of producing K\"ahler-Einstein orbifold metrics is the
work of Demailly and Koll\'ar \cite{DeKo01}. ÊThey establish the existence of K\"ahler-Einstein metrics in our orbifold setting unless a certain analytic multiplier ideal sheaf is non-trivial. ÊThe multiplier ideal sheaves are associated to pluri-subharmonic functions, but there is
a dictionary (see  pp. 176-178 of \cite{Laz04b}, and pp. 157-158 of \cite{BG05} for the orbifold setting)
translating this into algebraic language. ÊIn particular, if the
pluri-subharmonic function $\varphi$ locally defines a divisor $D$ on $\calz_f$ that is invariant under the local uniformizing groups $\grG$,
then triviality  of the analytic multiplier ideal sheaf is equivalent
to a certain algebraic condition known as {\it Kawamata log-terminal (klt)}. More precisely, let $D=\sum_ia_iD_i$ be described by local $\grG$-invariant holomorphic functions $g_i$. Then the function $\varphi_D=\sum_ia_i\log |g_i|$ is
plurisubharmonic and defines a multiplier ideal sheaf
$$\cali(\varphi_D)_{\tU}= \Bigl\{f\in \calo_X(\tU)^{\grG}~|~\int_{\tU}\frac{|f|^2}{\prod
|g_i|^{2a_i}}d{\rm vol}_g<\infty\Bigr\}$$ on the local uniformizing neighborhood $\tU,$ and, hence, a multiplier ideal orbisheaf $\cali(\calz,D)$ on $\calz_f$.

\begin{definition} \label{klt2}
The pair $(\calz,D)$ is {\bf klt} or {\bf Kawamata log-terminal}) if
for each local uniformizing neighborhood $\tU$ there exists a log
resolution of singularities $\mu: X \rightarrow \tU$ such that
$$
K_{X} \equiv_n \mu^*(K^{orb}_{\tU} + \phi^*D) + \sum a_iE_i
$$
with $a_i > -1$ for all $\mu$-exceptional $E_i.$  
\end{definition}

This is equivalent to the non-triviality of the multiplier ideal orbisheaf, that is, the condition $\cali(\calz,D)=\calo_\calz.$
Thus Demailly and Koll\'ar prove the existence of a K\"ahler-Einstein orbifold metric provided all divisors numerically equivalent to an appropriate multiple of $-K_\calz^{orb}$ are klt. When phrased in terms of holomorphic sections this becomes Equation (\ref{l2eqn}).

In Corollary 6.1 of  \cite{BGN03c} the authors gave the following result when the weight vector $\bfw$ is {\it well-formed}, that is the gcd of any $n$ of the weights $w_0,\cdots,w_n$ is one. This corollary  is derived from Corollary 4.1 of
the same paper.  As has been kindly pointed out to the authors
by Ivan Cheltsov, there is a gap in the proof of Corollary 4.1
as the argument sketched there only applies at the singular
points of $\bbp(\bfw)$ and a different technique,
such as the semi-continuity argument of Johnson and Koll\'ar
given in Proposition 11 of \cite{JoKo01a}, is required to deal with
all points on $\calz_f$.  So here we sketch a proof which fills this gap. Moreover, in  \cite{BGN03c} the authors only treated the case when $\bfw$ is well-formed; however, the corrected argument works equally well in the non-well-formed case as long as one replaces the canonical divisor $K_{\calz_f}$ by the orbifold canonical divisor  $K_{\calz_f}^{orb}=K_{\calz_f}+\grD$ where is a branch divisor, see \cite{BGK05} and Chapter 4 of \cite{BG05}.

\begin{lemma}\label{klt5}
Let $L(\bfw,d)$ be a link of a weighted homogeneous hypersurface with weight vector $\bfw=(w_0,w_1,w_2,w_3)$ ordered as $w_0\leq w_1\leq w_2\leq w_3.$ Let $\calz_\bfw$ denote the corresponding projective algebraic orbifold. Furthermore, let $I=|\bfw|-d$ denote the Fano index. Then
\begin{enumerate}
 \item The 5-manifold $L(\bfw,d)$ admits a Sasaki-Einstein metric if $2Id<3w_0w_1.$ 
\item If the line $z_0=z_1=0$ does not lie in $\calz_\bfw$ and the weaker condition $2Id<3w_0w_2$ holds, then $L(\bfw,d)$ admits a Sasaki-Einstein metric.
\item If the point $(0,0,0,1)$ does not lie in $\calz_\bfw$ and the even weaker condition $2Id<3w_0w_3$ holds, then $L(\bfw,d)$ admits a Sasaki-Einstein metric.
\end{enumerate}
\end{lemma}

\begin{proof}
We consider Proposition 11 of \cite{JoKo01a} only in the situation
needed for this paper, namely where $\calz_\bfw \subset  
\bbp(\bfw)$
is a hypersurface in weighted projective three--space.  We will assume
that $\bfw=(w_0,w_1,w_2,w_3)$ satisfies $w_0 \leq w_1 \leq w_2 \leq w_3$ and $\gcd(w_0,w_1,w_2,w_3)=1.$
Let $\phi: C_f\backslash \{0\} \rightarrow \bbp(\bfw)$ be the natural quotient map
and let $Z_i = \phi^{-1}(\calz_f \cap \bbp(\bfw)\backslash\{z_i = 0\})$. 
Proposition 11 of \cite{JoKo01a} states that for each $p \in Z_i$
$$
\mult_p(Z_i) \leq w_3w_2w_1 \deg(\calz_\bfw)
$$
where the degree is computed relative to $\calo_{\bbp(\bfw)}(1)$.
Moreover, as long as $\calz_\bfw$ is not the hyperplane $z_0 = 0$ then there is a
better bound:
$$
\mult_p(Z_i) \leq w_3w_2w_0 \deg(\calz_\bfw)
$$
for all $i$ and all $p \in Z_i$.
Finally (this is the extra content of Corollary 4.1 of \cite{BGN03c}),
if $(0,0,0,1)$ is not contained in $\calz_\bfw$ then for all $i$ and all
$x \in Z_i$
$$
\mult_p(Z_i) \leq w_2w_1w_0 \deg(\calz_\bfw).
$$

One key idea in the proof of Proposition 11 of [JK01] is to use the
affine cone $C_f$ lying over $Z_f$ in order to bound the multiplicity of $Z_f$ at various points.
Since the multiplicity of a variety at a point is an upper semi--continuous
function of the point we deduce that for any point $p \in Z_i$
$$
\mult_0(C_f) \geq \mult_p(Z_i)
$$
and thus the goal is to bound the multiplicity of $C_f$ at
the origin.

There are two steps in the process of bounding $\mult_0(C_f)$.  First
we cut out a zero dimensional cycle $Z$ by successively intersecting $\calz_\bfw$
with the coordinate hyperplanes $H_i = \{z_i = 0\}$.  In general, different
hyperplanes are required for different components of the intersection class
but in our case the entire process works globally.
Lifting the cycle $Z$ to $C_f$, this
corresponds to cutting out the closure $\bar{F}_x$ of each
 fibre $F_x = \phi^{-1}(x)$ lying over a point $x \in Z$, with ramification
index $r_x$ given by the order of the kernel of the action of ${\bf C}^\ast$ 
on $F_x$.  
As to the multiplicity of these fibres at $0$ which we must bound, we have
\begin{eqnarray}
r_x \cdot \mult_0(\bar{F}_x) = w_{\min_{j}\{z_j \neq 0\}}
\label{mult}
\end{eqnarray}
as can be seen by looking at the relevant affine chart in
the definition of $\bbp(\bfw)$ as described in \S 1 of \cite{JoKo01a}.

Suppose, for the moment, that the point $(0,0,0,1)$  does not belong
to $\calz_\bfw$.  We claim that $H_0 \cap H_1 \cap \calz_\bfw$ is a proper intersection.
Since $H_0 \cap H_1 \subset \bbp(\bfw)$  is a line, if the 
intersection were not proper then the entire line $H_0 \cap H_1$ would
be contained in $\calz_\bfw$ and this is not the case since $(0,0,0,1)$ is
not in $\calz_\bfw$.  Note that a weaker hypothesis is sufficient to guarantee
that $H_0 \cap H_1  \cap \calz_\bfw$ is proper but we will use the fact that
$(0,0,0,1)$ is not in $\calz_\bfw$ at the next step.  Let $x$ be one of the
points of intersection of $H_0 \cap H_1 \cap \calz_\bfw$.  Then
\begin{eqnarray}
r_x \cdot \mult_0(\bar{F}_x)  \leq w_2
\label{ub}
\end{eqnarray}
since $x$ must be contained in one of the charts $z_i \neq 0$ for 
$i \leq 2$ and so (\ref{ub}) follows from (\ref{mult}).
Since $\calz_\bfw \cap H_0 \cap H_1$ is a cycle of degree at most $\deg(\calz_\bfw)w_0w_1$ we see
from (\ref{ub}), lifting to $C_f$, that
$$
\mult_0(C_f) \leq \deg(\calz_\bfw)w_0w_1w_2
$$
as desired; indeed, the right hand side is the total degree of an intersection
class, part of which, the left hand side, is supported at $0$.  
The other two inequalities of \cite{JoKo01a} Proposition 11 are obtained
similarly except that in these cases one cannot necessarily
choose $H_0$ and $H_1$ for the intersection and one does not 
necessarily have the inequality (\ref{mult}) and hence the
inequalities obtained are weaker.  This proves the lemma as well as fills the gap in 
Corollary 6.1 of \cite{BGN03c}.

\end{proof}

Note that all BP polynomials satisfy the better estimate (iii) of Lemma \ref{klt5}. However, in \cite{BGK05} a different estimate was used for the case of BP polynomials. These two estimates are somewhat independent. The estimate used in \cite{BGK05} appears to be the better estimate for spheres, but the estimate in (iii) of Lemma \ref{klt5} seems to be the better for Smale manifolds of mixed type as well as some rational homology spheres.

We also remark that many of the Sasaki-Einstein metrics given here occur with `moduli'. Using methods explained in \cite{BG05}, it is straightforward to compute infinitesimal deformations of a Sasaki-Einstein structure on a given link. However, completeness of this infinitesimal deformation space is still unknown, and will be treated elsewhere.

\section{Smale manifolds}

The reason that dimension five is so amenable to analysis stems from Smale's seminal work \cite{Sm62} on the classification of compact simply connected spin 5-manifolds. The non-spin case was later completed by Barden \cite{Bar65}, but the spin case will suffice here. Smale's classification of all closed simply 
connected $5$-manifolds that admit a spin structure is as follows. Any such manifolds must be of the form
\begin{equation}\label{smaleman}
M=kM_\infty\# M_{m_1}\#\cdots \#M_{m_n}
\end{equation} 
where $M_{\infty}= S^2\times S^3$, 
$kM_{\infty}$ is the $k$-fold connected sum of $M_{\infty}$,   
$m_i$ is a positive integer with $m_i$ dividing $m_{i+1}$ and $m_1\geq 1$, 
and where $M_{m}$ is $S^5$ if $m=1$, or a 
$5$-manifold such that $H_2(M_m,{\mathbb Z})={\mathbb Z}/m \oplus 
{\mathbb Z}/m$, otherwise. 
Here $k\in \bbn$ and $k=0$ means that there is
no $M_{\infty}$ factor at all. It will also be convenient to
use the convention $0M_{\infty}=S^5$, which is consistent
with the fact that the sphere is the identity element for the connected
sum operation. 

The understanding of Sasakian geometry in dimension five owes much to the work of Koll\'ar \cite{Kol05b,Kol06a,Kol06}. In particular, he has given a complete characterization of the torsion \cite{Kol05b} for smooth Seifert bundles over a projective cyclic orbifold surface. Explicitly, he has shown that if a Smale manifold $M$ admits a Sasakian structure, then $H_2(M,\bbz)$ must take the form
\begin{equation}\label{Koltor}
H_2(M,\bbz)=\bbz^{k}\oplus\sum_i
(\bbz_{m_i})^{2g(D_i)},
\end{equation}
where $g(D)$ denotes the genus of the branch divisor $D.$ This shows that only non-rational branch divisors contribute to torsion. To determine the Smale manifold we can also use a combinatorial method due to Orlik \cite{Or72} and a computer program written for us by E. Thomas. We refer to Corollary 9.3.13 and Conjecture 9.3.15 of \cite{BG05} for complete details. This conjecture is known to hold in dimension five \cite{BGS07b}.

\section{Torsion and Branch Divisors}

Heretofore, $L_f=L_f(\bfw,d)$ will denote a 5 dimensional link of a weighted homogeneous polynomial $f$ of degree $d$ and weight vector $\bfw=(w_0,w_1,w_2,w_3).$ As mentioned previously such five dimensional links have been classified by Yau and Yu \cite{YaYu05}. They gave a list of 19 normal forms which is reproduced in the Appendix. First we make note of the fact that 10 of the 19 types of Yau-Yu links can be written as a linear system of 4 equations in 4 unknowns. These are types I-VII, X,XI, and XIX. All of these can be written as a matrix equation in the form
\begin{equation}\label{linsys1}
{\bf A}\left( 
\begin{matrix}w_0 \\ w_1 \\ w_2 \\ w_3
\end{matrix}\right) =d\left(
\begin{matrix} 1 \\ 1 \\ 1 \\ 1
\end{matrix}\right),
\end{equation}
where ${\bf A}$ is 4 by 4 matrix depending on the exponents $\bfa.$ We refer to ${\bf A}$ as the {\it exponent matrix}. We assume that $\det~{\bf A}\neq 0,$ in which case we write the solution as $\bfw =d{\bf A}^{-1}\BOne_4.$ We refer to links of this type as {\it standard YY links}. We are interested in 5-manifolds $M^5$ that have non-zero torsion, and by Koll\'ar's result Equation (\ref{Koltor}) above, non-trivial torsion implies there must be branch divisors that are non-rational curves.  We now consider some useful lemmas.

\begin{lemma}\label{torsion}
Let $\calz_\bfw$ be the algebraic orbifold of a link $L_f=L_f(\bfw,d)$. Then any branch divisor $D$ that contributes to torsion in $H_2(L_f,\bbz)$ is a coordinate hypersurface in $\calz_\bfw,$ that is, $D$ is the hypersurface $z_i=0$ for some $i=0,\cdots,3.$ In particular, $\calz_\bfw$ can have at most four branch divisors contributing to torsion.
\end{lemma}

\begin{proof}
This follows from Koll\'ar's result (\ref{Koltor}) above and the discussion on page 141 of \cite{BG05}. Indeed, if $D$ is a branch divisor contributing to torsion then the algebraic orbifold $\calz_\bfw$ is not well-formed, so some $d_i=\gcd(w_0,\cdots,\hat{w_i},\cdots,w_3)>1$, and $D$ is obtained by setting $z_i=0.$ 
\end{proof}

\begin{lemma}\label{stlem1}
Let $L_f=L_f(\bfw;d)$ be a link such that the polynomial $f$ contains a term of the form $z_j^{a_j}z_i.$ Then $D_i$ is not a branch divisor. 
\end{lemma}

\begin{proof}
Since the polynomial $f$ contains a term of the form $z_iz_j^{a_j},$ we must have $d=a_jw_j+w_i.$ But if $D_i$ were a branch divisor, we would have $m_i=\gcd(w_j: j\neq i)>1$. But then since $m_i|d$ and all the $w_j$ with $j\neq i,$ it must divide $w_i=d-a_jw_j$ as well. But this is impossible since $\gcd(w_0,w_1,w_2,w_3)=1.$
\end{proof}

Using Lemmas \ref{torsion} and \ref{stlem1} we give a table of the number $N$ of possible branch divisors that can contribute to torsion for each of the standard YY types.
\medskip
\begin{center}
\begin{tabular}{|c||c|c|c|c|c|c|c|c|c|c|}\hline\hline
YY type &I&II&III&IV&V&VI&VII&X&XI&XIX \\ \hline
$N$ &4&3&2&2&1&0&2&1&1&0 \hl
\end{tabular}
\vspace{3mm}\\
\end{center}
Moreover, it follows from Proposition 6.3 of \cite{Kol05b} that a positive Sasakian structure can have at most one non-rational branch divisor.

The following result about 3-dimensional links will prove useful.
\begin{lemma}\label{3dlem}
Any curve $C$ of the form $z_0^{a_0}+z_1^{a_1}+ z_1z_2=0$ or $z_0^{a_0}+z_1z_2=0$ is rational.
\end{lemma}

\begin{proof}
In both cases we have $d=w_1+w_2.$
So $|\bfw|-d=w_0>0,$ which implies that $C$ is rational by a theorem of Orlik \cite{Or70}.
\end{proof}

Instead of presenting a systematic investigation of all the YY links which would require computer searches, we prefer here to give some methods of ferreting out links with torsion in the hope of providing at least partial solutions to the problems given in the introduction. We also restrict ourselves at this time to the standard YY links. A more systematic approach using computer searches is currently in progress. Here we present two methods which are complementary in the sense that one begins with the case of many (meaning four) branch divisors, namely links of BP polynomials (which are classified in \cite{BG05}) with fixed exponents and relatively high index and changes (a procedure we call {\it jiggling}) the weight vector to obtain a different YY type. The second method begins with a YY type with few branch divisors, and tries to give conditions to guarentee non-trivial torsion. This can also give links with more branch divisors as `degenerate cases'. The minimal YY type is always presented. We illustrate these procedures with examples.

\begin{example}
Consider the BP link of degree $24$ and exponent vector $\bfa=(2,3,8,8)$ given in Table B.4.1 of \cite{BG05}. This gives a positive Sasaki-Seifert structure on the rational homology sphere $3M_3$ which also admits a Sasaki-Einstein metric. Its weight vector is $\bfw=(12,8,3,3)$, and its degree is $d=24.$ We see that $D_2$ is a branch divisor with ramification index $3$ and genus $3.$ Since the index $I=|\bfw|-d=2,$ we can jiggle the weight vector by lowering the weight $w_1$ to $7$ and replacing the monomial $z_1^3$ by $z_1^3g_1(z_2,z_3)$ where $g_1$ is a linear polynomial in $z_2$ and $z_3.$ This gives the new link of YY type II with weight vector $\bfw=(12,7,3,3),$ degree $d=24,$ and index $I=1$ on the Smale manifold $M_\infty\#3M_3.$ Moreover, one sees by Lemma \ref{klt5} that this admits a Sasaki-Einstein metric.
\end{example}

\begin{example}
For the second procedure we consider a YY link of type X.
In this case the normal form is 
$$z_0^{a_0}+z_1^{a_1}z_2+z_2^{a_2}z_3 +z_1z_3^{a_3}.$$
The positivity condition is
$$\frac{1}{a_0} +\frac{a_1a_2+a_1a_3+a_2a_3-a_1-a_2-a_3+3}{a_1a_2a_3+1}>1.$$
The only possible branch divisor is $D_0$ which we want to be a non-rational curve. Notice that if we denote the degree of the divisor $D_0$ by $d'$ with weights $\bfw'=(w_1',w_2',w_3'),$ then a necessary condition for $D_0$ to be non-rational is $d'-|\bfw'|\geq 0.$ Now we write $\bfw=(w_0,mw_1',mw_2',mw_3')$ with $\gcd(w_1',w_2',w_3')=1.$ We also have 
\begin{equation}\label{w'eq}
d=a_0w_0=a_1mw_1'+mw_2'=a_2mw_2'+mw_3'= a_3mw_3'+mw_1'.
\end{equation}
So $m$ divides $d$ and $\gcd(m,w_0)=1$, so we write $d=md'$ and $a_0=ma_0'.$ The general solution for the weight vector $\bfw'$ of the branch divisor $D_0$ is
\begin{equation}\label{genX}
\bfw'=\left(
\begin{matrix} w_1' \\
               w_2' \\
               w_3'
\end{matrix}\right)=
\frac{d'}{1+a_1a_2a_3}
\left(\begin{matrix} 1-a_3+a_2a_3  \\
                     1-a_1+a_1a_3\\
                     1-a_2+a_1a_2
              \end{matrix}\right).
\end{equation}
We see from Equation (\ref{w'eq}) that $\gcd(w_i',d')=1$ and $\gcd(w_i',w_j')=1$ for all $i,j.$ Moreover,  from (\ref{genX}) and the fact that $w_i'$ are relatively prime integers, we must have $qd'=1+a_1a_2a_3$ where $q$ is a positive integer that must divide $b_i:=a_ia_{i+1}-a_{i+1}+1$ where $i\in \{1,2,3\}\mod~3.$
Now using the genus formula (Equation 4.6.5 of \cite{BG05}), and the general solution (\ref{genX}) gives, with the help of Maple,
\begin{equation}\label{Xgenus}
g(D_0)=\frac{1+a_1a_2a_3-d'}{2d'}=\frac{q-1}{2}
\end{equation}
which must be a non-negative integer implying that $q$ must be odd.

We consider the case $\bfa=(2,3,5,a_3)$ which also implies $m=2.$ We look for solutions with $q>1$ and odd, of course. It is now quite easy to satisfy the positivity condition which gives $5\leq a_3\leq 18.$ Now searching for possible $b_i,$ we see that the only solution is for $a_3=8,$ in which case we have $b_1=11,b_2=33$ and $b_3=22$ giving $q=11$. We find $d=22$ and $\bfw=(2,4,6,11)$ which is a perturbation of type VII, and occurs on the manifold $5M_2 $. It satisfies the optimal klt condition of Lemma \ref{klt5}, and thus, admits a Sasaki-Einstein metric. 
\end{example}

\section{Proofs of Theorems A and B}
Given our previous discussion the proofs of the main theorems are now straightforward. We need only exhibit links $L(\bfw,d)$ of isolated hypersurface singularities by weighted homogenous polynomials of degree $d$ and weight vector $\bfw,$ and then check the positivity condition $I=|\bfw|-d>0,$ and the klt estimates of Lemma \ref{klt5}. We list the results in the appendix in the form of tables. Since we have come up short of a complete classification, we do not bother keeping track of the number of Sasaki-Seifert structures that occur when the cardinality is finite. We do, however, record the cases when there are a countable infinity of Sasaki-Seifert structures.

\section{Appendix: Tables}

Here we present tables giving links of isolated hypersurface singularities by weighted homogeneous polynomials. In the tables we list the weight vector $\bfw$, the degree $d$, the Fano index $I=|\bfw|-d$ of the weighted homogeneous polynomial, the manifold, and whether there is a Sasaki-Einstein metric. In the case of the sporadic solution, we also list the exponent vector $\bfa$ of a typical polnomial as well as the Yau-Yu normal form type \cite{YaYu05} (see also Table B.5 of \cite{BG05}). We remark that the exponent vector is not generally an invariant of  the Sasaki-Seifert structure, nor of the link, so we list the $\bfa$ that `minimizes'  the YY type. Moreover, we order the weights as $w_0\leq w_1\leq w_2\leq w_3.$ 

First we exhibit a table of infinite series solutions that admit Sasakian metrics of positive Ricci curvature, and will thus prove Theorem B. Four of the series admit Sasaki-Einstein metrics proving part (3) of Theorem A, while 3 series cannot admit Sasaki-Einstein metrics when $k>4.$ The existence of such metrics is unknown in all other cases. It is interesting that the only series solutions that we found that admit Sasaki-Einstein metrics all involve $M_3.$ Note that here $(a,b)=\gcd(a,b).$

\begin{center}\vbox{\[\begin{array}{|c|c|c|c|c|} \hline  {\rm weight ~vector}~\bfw
& d &I &{\rm manifold} &SE \hl\hline 
(2,2k,k(2n+1),2k(n+1)),~ (k,2)=1 &4k(n+1) &k+2 &M_\infty\#nM_2 &no,~if~ k>4 \hl
(2,4,4n,4n+1) &4(2n+1) &3 &2M_\infty\#nM_2 &? \hl
(4,4k+2,4k+3,2(4k+3)) &4(4k+3) &3 &2M_\infty\#M_2 &? \hl
(4,4k-1,4k,8k) &16k &3  &M_\infty\#M_4  &?\hl
(2,3k+2,4k+2,2(3k+2)),~ (k,2)=1 &4(3k+2) &k+2 & 3M_\infty\# M_2 &no,~if~ k>4\hl
(3,3k+2,3(2k+1),3(3k+2))&6(3k+2) &2 & 3M_\infty\# M_3 &yes \hl 
(6,3(2k+1),4(3k+1),9(2k+1)) &18(2k+1) &4 & M_\infty\#M_3 &yes \hl
(3,3k,6k-1,9k) &18k &2 &M_\infty\#2M_3 &yes \hl
(3,3k-1,3k,3k) &9k &2  &2M_\infty\#M_3 &? \hl
(2,2k+1,2(2k+1),2(3k+1)) &6(2k+1) &1 &6M_\infty\#M_2 &? \hl
(2,2k,4k,6k-1) &12k &1 &2M_\infty\#4M_2 &? \hl
(6,2k,4k,3(2k-1)), ~ (k,3)=1 &12k &3  &2M_\infty\#M_2 &? \hl
(2,2k,2k,4k-1) &8k &1 &3M_\infty\#3M_2 &? \hl
(3,3k,6k-1,9k) &18k &2 &2M_3 &yes \hl
(4,3(2k+1),4(2k+1),4(3k+1)) &12(2k+1) &2k+3  &M_4 &no,~if~ k>4\hl
(6,6k-1,12k,18k) &36k &5  &M_6 &?\hl
\end{array}\]}
\vspace{3mm}
\parbox{4.00in}{\small Table 1. Series solutions giving a countable infinity of positive Sasaki-Seifert structures on manifolds having non-vanishing torsion.}
\end{center}

The next two tables will complete the proof of Theorem A. The first table gives a list of rational homology spheres of the form $nM_2$ that can admit Sasaki-Einstein metrics. None of these rational homology spheres with $n>1$ were known previously to admit such metrics. We also indicate some links where the existence of Sasaki-Einstein metrics is still unknown.

\begin{center}\vbox{\[\begin{array}{|c|c|c|c|c|c|c|} \hline \bfa & \bfw 
& d& I &{\rm manifold} &SE & {\rm YY type} \hl\hline 
(11,5,3,2) &(2,4,6,11) &22 &1 &5M_2 &yes &VII \hl
(7,7,3,2)  &(2,2,4,7)  &14 &1 &6M_2 &yes & II  \hl
(27,7,3,2)  &(4,18,42,63) &126 &1 &3M_2 &yes & II \hl
(14,9,3,2)  &(8,14,42,63) &126 &1 &M_2 &yes & II \hl
(9,7,3,2)  &(4,6,14,21) &42 &3 &3M_2 &? & II \hl
(9,7,3,2)  &(4,6,14,19) &42 &1 &3M_2 &yes & VII \hl
(11,9,3,2) &(18,20,66,99) &198 &5 &M_2 &? & II \hl
(11,9,3,2) &(18,22,60,99) &198 &1 &M_2 &yes & II \hl
(10,9,3,2) &(8,10,30,45) &90 &3 &M_2 &yes & II\hl
(10,9,3,2) &(6,10,30,45) &90 &1 &M_2 &yes & II\hl
(15,7,3,2) &(2,4,10,15) &30 &1 &6M_2 &yes & II \hl
(12,9,3,2) &(4,6,18,27)  &54 &1 &3M_2 &yes & II\hl
(17,4,3,2) &(4,14,26,41) &82 &3 &2M_2 &? &X \hl
(12,4,3,2) &(4,10,18,29) &58 &3 &2M_2 &? &X \hl
(7,4,3,2) &(4,6,10,17) &34 &3 &2M_2 &? &X \hl
(15,5,3,2) &(2,6,8,15) &30 &1 &5M_2 &yes &II \hl
(5,5,5,2) &(2,2,2,5) &10 &1 &6M_2 &yes &I \hl
(6,5,5,2) &(4,6,6,15) &30 &1 &2M_2 &yes &II \hl
(7,5,5,2) &(10,12,14,35) &70 &1 &2M_2 &yes &II\hl
(9,9,3,2) &(2,2,6,9) &18 &1 &7M_2 &yes &I \hl
(21,7,3,2) &(2,6,14,21) &42 &1 &6M_2 &yes &I \hl
(15,9,3,2) &(6,10,30,45) &90 &1 &M_2 &yes &I \hl

\end{array}\]}
\vspace{3mm}
\parbox{4.00in}{\small Table 2. Sporadic examples of rational homology spheres of the form $nM_2$ with $n>1$ that can admit a Sasaki-Einstein metric.} 
\end{center}
\vfil\eject

Finally, we give a table of Smale manifolds of mixed type most of which admit Sasaki-Einstein metrics, thus completing the proof of Theorem A. 

\medskip

\begin{center}\vbox{\[\begin{array}{|c|c|c|c|c|c|c|}
\hline
\bfa & {\bf w} & d &I & {\rm manifold} & SE  &YY type\\
\hline\hline
(5,3,3,3) & (9,10,12,15) & 45 & 1 & M_\infty\#M_3 & yes & IV \hl
(8,4,3,2) & (3,6,7,9) & 24 & 1 & M_\infty\#2M_3 & yes & IV \hl
(9,4,3,2) & (2,4,6,7) & 18 & 1 & 2M_\infty\#3M_2 & ? & VII \hl
(10,4,3,2) & (4,10,12,15) & 40 & 1 & 3M_\infty\#M_2 & yes & IV \hl
(6,4,4,2) & (8,9,12,20) & 48 &1 & M_\infty\#M_4 & yes &IV \\
\hline
(16,5,3,2)&(3,9,13,24) &  48 &1 &M_\infty\#2M_3 &yes &VII \hl
(6,4,2,2) & (4,6,9,10) & 24 &5 & M_\infty\#M_2 & ? &IV \\
\hline
(6,4,4,2) & (4,6,6,9) & 24 &1 & M_\infty\#M_2 & yes &IV\\
\hline
(6,5,3,2) & (10,12,16,25) & 60 &3 & M_\infty\#M_2 &? &IV\\
\hline
(12,5,3,2) & (10,24,32,55) & 120 &1 & M_\infty\#M_2 & yes &IV\\
\hline
(6,4,3,3) &(4,5,8,8) &24 & 1 &2M_\infty\#M_4 &yes &II\hl
(8,6,4,2) &(6,7,12,24)  &48  &1 &M_\infty\#M_6 &yes &II  \hl
(12,4,4,2)&(4,9,12,24) &48 &1 &M_\infty\#M_4 &yes &II \hl
(7,6,4,2) &(6,7,9,21)  &42 &1 &2M_\infty\#M_3 &yes &II \hl 
(6,6,4,2) &(4,4,5,12)  &24 &1 &M_\infty\#2M_4 &yes &II\hl
(12,5,4,2) &(5,11,15,30) &60 &1 &M_\infty\#M_5 &yes &II \hl
(12,7,3,2) &(7,8,28,42) &84 &1 &M_\infty\#M_7 &yes  &II \hl
(10,8,3,2) &(4,5,12,20)  &40  &1  &3M_\infty\#M_4 &yes &II  \hl
(8,8,3,2) &(3,3,7,12)  &24 &1 &M_\infty\#3M_3 &yes &II \hl
(21,6,3,2)&(2,7,14,20) &42 &1 &6M_\infty\#M_2 &? &II \hl
(8,4,4,2)&(2,3,4,8) &16 &1 &5M_\infty\#M_2 &yes &II \hl
(6,6,4,2) &(2,2,3,6) &12 &1 &5M_\infty\#2M_2 &yes &I \hl
(12,8,3,2) &(2,3,8,12) &24 &1 &6M_\infty\#M_2 &yes &I\hl
(10,5,4,2) &(2,4,5,10) &20 &1 &4M_\infty\#2M_2 &yes &I \hl
(12,9,3,2) &(3,4,12,18) &36 &1 &4M_\infty\#M_3 &yes &I\hl
(8,6,4,2) &(3,4,6,12) &24 &1 &3M_\infty\#M_3 &yes &I \hl
(12,10,3,2) &(5,6,20,30) &60 &1 &2M_\infty\#M_5 &yes &I \hl
(18,8,3,2) &(4,9,24,36) &72 &1 &2M_\infty\#M_4 &yes &I \hl
(9,6,4,2) &(4,6,9,18) &36 &1 &2M_\infty\#M_2 &yes &I \hl
(12,6,3,2)&(3,6,10,18) &36 &1&M_\infty\#2M_3 &yes &II \hl
(12,6,3,2)&(3,6,11,18) &36 &2&M_\infty\#2M_3 &yes &II \hl
(6,6,3,2) &(3,3,5,9) &18 &2 &M_\infty\#2M_3 &yes &II \hl
(20,5,4,2)&(3,12,16,30)&60 &1&M_\infty\#2M_3 &yes &II \hl
(8,5,4,2)&(5,7,10,20) &40 &2 &M_\infty\#M_5 &yes &II \hl
(8,5,4,2)&(5,6,10,20) &40 &1 &M_\infty\#M_5 &yes &II \hl
\end{array}\]}
\vspace{3mm}
\parbox{4.00in}{\small Table 3. Some new sporadic examples of simply connected 5-dimensional manifolds of mixed
type which can admit Sasaki-Einstein metrics.}

\end{center}
\vfil\eject

\begin{center}
\begin{tabular}{|l|c|c|}\hline
\multicolumn{3}{|c|}{{\bf Table}: The Yau-Yu Links in  Dimensions 5 }\\
\hline
{\bf }Type & $f(z_0,z_1,z_2,z_3)$ & $|\bfw|/{\rm degree}(f)$\\
\hline\hline
{\bf I} (BP)& $z^{a}_0+z^{b}_1+z^{c}_2+z^{d}_3$&
$\frac{1}{a}+\frac{1}{b}+\frac{1}{c}+\frac{1}{d}$\\
\hline
{\bf II} & $z^{a}_0+z^{b}_1+z^{c}_2+z_2z^{d}_3$ &
$\frac{1}{a}+\frac{1}{b}+\frac{1}{c}+\frac{c-1}{cd}$\\
\hline
{\bf III} & $z^{a}_0+z^{b}_1+z^{c}_2z_3+z_2z^{d}_3$ &
$\frac{1}{a}+\frac{1}{b}+\frac{d-1}{cd-1}+\frac{c-1}{cd-1}$\\
\hline
{\bf IV} & $z^{a}_0+z_0z^{b}_1+z_2^{c}+z_2z^d_3$ &
$\frac{1}{a}+\frac{a-1}{ab}+\frac{1}{c}+\frac{c-1}{cd}$\\
\hline
{\bf V}& $z^{a}_0z_1+z_0z^{b}_1+z^{c}_2+z_2z_3^{d}$ &
$\frac{b-1}{ab-1}+\frac{a-1}{ab-1}+\frac{1}{c}+\frac{c-1}{cd}$\\
\hline
{\bf VI} &  $z^{a}_0z_1+z_0z^{b}_1+z^{c}_2z_3+z_2z_3^{d}$&
$\frac{b-1}{ab-1}+\frac{a-1}{ab-1}+\frac{d-1}{cd-1}+\frac{c-1}{dc}$\\
\hline
{\bf VII} & $z^{a}_0+z^{b}_1+z_1z^{c}_2+z_2z_3^{d}$&
$\frac{1}{a}+\frac{1}{b}+\frac{b-1}{bc}+\frac{[b(c-1)+1]}{bcd}$\\
\hline
{\bf VIII}& $z^{a}_0+z^{b}_1+z_1z^{c}_2+z_1z_3^{d}+z_2^pz_3^q$ &
$\frac{1}{a}+\frac{1}{b}+\frac{b-1}{bc}+\frac{b-1}{bd}$\\
& $\frac{p(b-1)}{bc}+\frac{q(b-1)}{bd}=1$& \\
\hline
{\bf IX} & $z^{a}_0+z^{b}_1z_3+z^{c}_2z_3+z_1z_3^{d}+z^p_1z^q_2$ &
$\frac{1}{a}+\frac{(d-1)}{bd-1}+\frac{b(d-1)}{c(bd-1)}+\frac{b-1}{bd-1}$\\
& $\frac{p(d-1)}{bd-1}+\frac{qb(d-1)}{c(bd-1)}=1$& \\
\hline {\bf X} & $z^{a}_0+z^{b}_1z_2+z^{c}_2z_3+z_1z_3^{d}$ &
$\frac{1}{a}+\frac{[d(c-1)+1]}{bcd+1}+\frac{[b(d-1)+1]}{bcd+1}$\\
& & $+\frac{[c(b-1)+1]}{bcd+1}$\\ \hline {\bf XI}
&$z^{a}_0+z_0z^{b}_1+z_1z^{c}_2+z_2z_3^{d}$ &
$\frac{1}{a}+\frac{a-1}{ab}+\frac{[a(b-1)+1]}{abc}$\\
& & $+\frac{[ab(c-1)+(a-1)]}{abcd}$\\ \hline {\bf XII}&
$z^{a}_0+z_0z^{b}_1+z_0z^{c}_2+z_1z_3^{d}+z^p_1z^q_2$ &
$\frac{1}{a}+\frac{a-1}{ab}+\frac{a-1}{ac}+\frac{[a(b-1)+1]}{abd}$\\
&$\frac{p(a-1)}{ab}+\frac{q(a-1)}{ac}=1$ &\\ \hline
{\bf XIII} & $z^{a}_0+z_0z^{b}_1+z_1z^{c}_2+z_1z_3^{d}+z^p_2z^q_3$ &
$\frac{1}{a}+\frac{a-1}{ab}+\frac{a-1}{ac}+\frac{[a(b-1)+1]}{abd}$\\
& $\frac{p[a(b-1)+1]}{abc}+\frac{q[a(b-1)+1]}{abd}=1$ & \\
\hline {\bf XIV} &
$z^{a}_0+z_0z^{b}_1+z_0z^{c}_2+z_0z_3^{d}+z^p_1z^q_2+z_2^rz_3^s$&
$\frac{1}{a}+\frac{a-1}{ab}+\frac{a-1}{ac}+\frac{a-1}{ad}$\\
&$\frac{p(a-1)}{ab}+\frac{q(a-1)}{ac}=1=\frac{r(a-1)}{ac}+\frac{s(a-1)}{ad}$ & \\
\hline {\bf XV}&
$z^{a}_0z_1+z_0z^{b}_1+z_0z^{c}_2+z_2z_3^{d}+z^p_1z^q_2$ &
$\frac{b-1}{ab-1}+\frac{a-1}{ab-1}+\frac{b(a-1)}{c(ab-1)}$\\
& $\frac{p(a-1)}{ab-1}+\frac{qb(a-1)}{c(ab-1)}=1$& $+\frac{[c(ab-1)-b(a-1)]}{cd(ab-1)}$\\
\hline {\bf XVI}&
$z^{a}_0z_1+z_0z^{b}_1+z_0z^{c}_2+z_0z_3^{d}+z^p_1z^q_2+z^r_2z^s_3$
& $\frac{(b-1)}{ab-1}+\frac{(a-1)}{ab-1}+\frac{b(a-1)}{c(ab-1)}
+\frac{b(a-1)}{d(ab-1)}$\\
&$\frac{p(a-1)}{ab-1}+\frac{qb(a-1)}{c(ab-1)}=1=\frac{r(a-1)}{ac}+\frac{s(a-1)}{ad}$ & \\
\hline {\bf XVII} &
$z^{a}_0z_1+z_0z^{b}_1+z_1z^{c}_2+z_0z_3^{d}+z^p_1z^q_3+z^r_0z^s_2$
& $\frac{b-1}{ab-1}+\frac{a-1}{ab-1}+\frac{a(b-1)}{c(ab-1)}
+\frac{b(a-1)}{d(ab-1)}$\\
&$\frac{p(a-1)}{ab-1}+\frac{qb(a-1)}{d(ab-1)}=1=\frac{r(b-1)}{ab-1}+\frac{sa(b-1)}{c(ab-1)}$ & \\
\hline {\bf XVIII}
&$z^{a}_0z_2+z_0z^{b}_1+z_1z^{c}_2+z_1z_3^{d}+z^p_2z^q_3$&
$\frac{[b(c-1)+1]}{abc+1}+\frac{[c(a-1)+1]}{abc+1}$\\
&$\frac{p[a(b-1)+1]}{abc+1}+ \frac{qc[a(b-1)+1]}{d(abc+1)}=1$ &
$+\frac{[a(b-1)+1]}{c(abc+1)}+\frac{c[a(b-1)+1]}{d(abc+1)}$\\
\hline {\bf XIX}& $z^{a}_0z_3+z_0z^{b}_1+z^{c}_2z_3+z_2z_3^{d}$ &
$\frac{[b(d(c-1)+1)-1]}{abcd-1}+\frac{[d(c(a-1)+1)-1]}{abcd-1}$\\
& &$+\frac{[a(b(d-1)+1)-1]}{abcd-1}+\frac{[c(a(b-1)+1)-1]}{abcd-1}$\\
\hline
\end{tabular}
\end{center}

\def\cprime{$'$} \def\cprime{$'$} \def\cprime{$'$} \def\cprime{$'$}
  \def\cprime{$'$} \def\cprime{$'$} \def\cprime{$'$} \def\cprime{$'$}
  \def\cdprime{$''$}
\providecommand{\bysame}{\leavevmode\hbox to3em{\hrulefill}\thinspace}
\providecommand{\MR}{\relax\ifhmode\unskip\space\fi MR }
\providecommand{\MRhref}[2]{%
  \href{http://www.ams.org/mathscinet-getitem?mr=#1}{#2}
}
\providecommand{\href}[2]{#2}

\end{document}